\documentclass[a4paper,12pt]{article}
\usepackage{amssymb}
\usepackage{mathrsfs}
\usepackage{amsmath}
\usepackage{amsfonts,amsthm,amssymb}
\usepackage{amsfonts}
\usepackage{graphics}
\usepackage{color}
\usepackage{graphicx}
\usepackage{epstopdf}
\usepackage{subfigure}

\newtheorem{lem}{Lemma}[section]
\newtheorem{thm}[lem]{Theorem}

\newtheorem{Conjecture}[lem]{Conjecture}

\begin{document}

\title{An Ore-type condition for existence of two disjoint cycles}
\author{Maoqun Wang and Jianguo Qian\footnote{Corresponding author. E-mail: jgqian@xmu.edu (J.G. Qian)}\\
\small School of Mathematical Sciences, Xiamen University, Xiamen 361005, PR China}
\date{}
\maketitle
{\small{\bf Abstract.}\quad  Let $n_{1}$ and $n_{2}$ be two integers with $n_{1},n_{2}\geq3$ and $G$ a graph of order $n=n_{1}+n_{2}$. As a generalization of Ore's degree condition for the existence of Hamilton cycle in $G$, El-Zahar proved that if $\delta(G)\geq \left\lceil\frac{n_{1}}{2}\right\rceil+\left\lceil\frac{n_{2}}{2}\right\rceil$ then $G$ contains two disjoint cycles of length $n_{1}$ and $n_{2}$. Recently,  Yan et. al considered the problem by extending the degree condition to degree sum condition and proved that if $d(u)+d(v)\geq n+4$ for any pair of non-adjacent vertices $u$ and $v$ of $G$, then $G$ contains two disjoint cycles of length $n_{1}$ and $n_{2}$. They further asked whether the degree sum condition can be improved to $d(u)+d(v)\geq n+2$. In this paper, we give a positive answer to this question. Our result also generalizes El-Zahar's result when $n_{1}$ and $n_{2}$ are both odd.
\vskip 0.3cm
\noindent{\bf Keywords:} degree sum condition; disjoint cycles

\section{Introduction}
We consider only finite simple graphs. For terminologies and notations not defined here we refer to \cite{Bondy}. Let $G$ be a graph with $n$ vertices. We denote by $V(G)$ and $E(G)$ the vertex set and edge set of $G$, respectively. For two vertices $u$ and $v$ in $G$, we denote by $uv$ the edge joining $u$ and $v$. A {\it neighbour} of a vertex $v$ is a vertex adjacent to $v$. The set of all  neighbours of $v$ is denoted by $N_G(v)$, or simply by $N(v)$, and the degree of $v$ is defined to be $d_G(v)=|N_G(v)|$. We use $\delta(G)$ to denote the minimum degree of vertices in $G$.  A path (resp., cycle) is called a {\it Hamilton path} (resp., {\it Hamilton cycle}) if it has length $n-1$ (resp., $n$). If $G$ has a Hamilton cycle,  then $G$ is called {\it Hamiltonian}. In general, if $G$ contains a cycle of length $k$ for every $k$ with $3\leq k\leq n$, then $G$ is called {\it pancyclic}. If, for any two distinct vertices $u$ and $v$, $G$ contains a Hamilton path with $u$ and $v$ as the two endvertices, then $G$ is called {\it Hamilton-connected}. Define
$$\sigma_{2}(G)=\min\{d(x)+d(y):x,y\in V(G),xy\notin E(G)\}.$$

The degree condition for the existence of cycle(s) with  specified length(s) is one of the most elementary concerns in graph theory. A classic result should be the one given by Dirac in 1952, which says that every graph of order $n$ with minimum degree at least $n$ has a Hamilton cycle. Since then, this result has been generalized to various forms in terms of degree condition or degree sum condition. We recall some typical results on this subject.

\begin{thm}\label{thm 1.1} (Dirac, \cite{Dirac})
For a graph $G$ of order $n$, if $n\geq3$ and $\delta(G)\geq\frac{n}{2}$ then $G$ is Hamiltonian.
\end{thm}

\begin{thm}\label{thm 1.2} (Ore, \cite{Ore})
For a graph $G$ of order $n$, if $n\geq3$ and $\sigma_{2}(G)\geq n$ then $G$ is Hamiltonian.
\end{thm}

\begin{thm}\label{thm A} (Bondy, \cite{Bondy2})
For a graph $G$ of order $n$, if $n\geq3$ and $\sigma_{2}(G)\geq n$ then $G$ is either pancyclic or else is the complete bipartite graph $K_{\frac{n}{2},\frac{n}{2}}$.
\end{thm}

\begin{thm}\label{thm 1.3} (El-Zahar, \cite{Zahar})
Let $n_{1},n_{2}$ be two integers with  $n_{1}, n_{2}\geq3$ and $G$ a graph of order $n$ with  $n=n_{1}+ n_{2}$. If $\delta(G)\geq \left\lceil\frac{n_{1}}{2}\right\rceil+\left\lceil\frac{n_{2}}{2}\right\rceil$ then $G$ has two disjoint cycles of length $n_{1}$ and $n_{2}$, where,  for a real number $r$, $\lceil r \rceil$ is the least integer not less than $r$.
\end{thm}

The degree condition in Theorem~\ref{thm 1.3} is sharp since the complete bipartite graph $K_{n/2,n/2}$ does not have any odd cycle. In general, El-Zahar posed the following conjecture:

\begin{Conjecture}\label{Conjecture 1.1}(El-Zahar, \cite{Zahar})
Let $G$ be a graph of order $n=n_{1}+ n_{2}+\cdots +n_{k}$ with $n_{i}\geq 3$ for each $i\in\{1,2,\cdots,k\}$. If $\delta(G)\geq \left\lceil\frac{n_{1}}{2}\right\rceil+\left\lceil\frac{n_{2}}{2}\right\rceil$$+\cdots + \left\lceil\frac{n_{k}}{2}\right\rceil$ then $G$ has $k$ disjoint cycles of length $n_{1}$, $n_{2},\cdots,n_{k}$.
\end{Conjecture}

In \cite{Abbasi}, Abbasi confirmed the conjecture for sufficiently large graphs by using the regularity lemma. For the special case when $n_{1}=n_{2}=\cdots=n_{k}=4$, the conjecture was posed  earlier by Erd\H{o}s \cite{E} and was proved later by Wang \cite{Wang}. In general, the conjecture still remains open. Recently, instead of degree condition, Yan et. al \cite{Yan} considered the problem from the view point of degree sum condition and proved the following result:

\begin{thm}\label{thm 1.4} (Yan. et.al,\cite{Yan})
Let $G$ be a graph on $n$ vertices. For any two integers $n_{1}$ and $n_{2}$ with  $n_{1}, n_{2}\geq3$ and $n=n_{1}+ n_{2}$, if $\sigma_{2}(G)\geq n+4$, then $G$ has two disjoint cycles of length $n_{1}$ and $n_{2}$.
\end{thm}
 As an extension of two disjoint cycles with specified lengths, Kostochka and Yu \cite{Kostochka} showed that if  $\sigma_{2}(G)\geq \frac{4}{3}n-1$ (not sharp in general) then the graph $G$ contains every 2-factor. On the other hand, the sharpness  of Theorem~\ref{thm 1.3} implies that the degree sum $\sigma_{2}(G)$ is at least $n+2$ for Theorem~\ref{thm 1.4}. Even so,  Yan et. al \cite{Yan} pointed that it might not be easy to get a better degree sum condition than the one in Theorem~\ref{thm 1.4} and therefore, posed the following question.

\noindent{\bf Question} \label{Problem 1.1} (Yan. et.al, \cite{Yan}).
 Let $G$ be a graph of order $n=n_{1}+ n_{2}$ with  $n_{1}, n_{2}\geq3$. Can $\sigma_{2}(G)\geq n+2$ guarantee that $G$ has two disjoint cycles of length $n_{1}$ and $n_{2}$?

In this paper, we give  a positive answer to this question.

\begin{thm}\label{thm 1.5}
Let $n_{1},n_{2}$ be two integers with  $n_{1}, n_{2}\geq3$ and $G$ a graph of order $n$ with  $n=n_{1}+ n_{2}$. If $\sigma_{2}(G)\geq n+2$ then $G$ has two disjoint cycles of length $n_{1}$ and $n_{2}$.
\end{thm}

 \section{Lemmas}

In order to prove the main theorem, in this section we introduce some necessary lemmas.

Let $W$ and $W'$ be two disjoint subsets of $V(G)$. We denote  by  $e(W,W')$ (resp., $e(W)$) the number of the edges in $G$ that lie between $W$ and $W'$ (resp., lie in $G[W]$), where $G[W]$ denotes the subgraph of $G$ induced by $W$. The graph obtained from $G$ by removing the vertices in $W$ is denoted by $G-W$. In particular, if $W$ consists of a single vertex $w$, then we simply write $e(\{w\},W')$ and $G-\{w\}$ as $e(w,W')$ and $G-w$, respectively.

For a cycle $C$ in $G$, we always give a direction on $C$ and use $C^{-}$ to denote the cycle $C$ with the opposite direction. For a vertex $v\in V(C)$, we use $v^{i-}$ and $v^{i+}$ to represent the $i$-th predecessor and $i$-th successor of $v$  along the direction of $C$, respectively. For simplicity, we write $v^{-}$ and $v^{+}$ instead of $v^{1-}$ and $v^{1+}$, respectively. For two vertices $u,v\in V(C)$, we denote by $C[u,v]$ the section of $C$ from $u$ to $v$ along the direction of $C$.

\begin{lem}\label{lem 2.1}\cite{Ore2}
 Let $u$ and $v$ be the two endvertices of a Hamilton path in a graph $G$ of order $n$. If $d(u)+d(v)\geq n$, then $G$ is Hamiltonian.
 \end{lem}

\begin{lem}\label{lem 2.2}\cite{Ore2}
 Let $C$ be a Hamilton cycle of a graph $G$ of order $n$  with a given direction and let $u$ and $v$ be two vertices on $C$. If $d_{G}(u^{+})+d_{G}(v^{+})\geq n+1 $ or $d_{G}(u^{-})+d_{G}(v^{-})\geq n+1 $, then $G$ contains a Hamilton path with endvertices $u$ and $v$.
\end{lem}

\begin{lem}\label{lem 2.3}\cite{Wang2}
 If $P$ is a path of length $k$ in a graph $G$ and $u, v$ are two vertices in $G-V(P)$ such that $e(\{u,v\}, V(P))\geq k+2$, then $G[V(P)\cup \{u,v\}]$ has a Hamilton path.
\end{lem}

\begin{lem}\label{lem 2.4}\cite{P}
 If $G$ is a graph on $n$ vertices and $\sigma_{2}(G)\geq n+1$, then $G$ is Hamilton-connected.
\end{lem}


\begin{lem}\label{lem 2.5}\cite{Yan}
 Let $n_{1}$ and $n_{2}$ be two integers with $n_{1}, n_{2}\geq5$ and let $G$ be a graph on $n=n_{1}+ n_{2}$ vertices with $\sigma_{2}(G)\geq n+2$. Suppose that $(G_{1},G_{2})$ is a partition of $G$ such that $V(G)=V(G_{1})\cup V(G_{2})$ and $|G_{i}|=n_{i}$ for i=1,2, and $G_{2}$ contains a Hamilton cycle. For two nonadjacent vertices $u,v\in V(G_{2})$, if $G_{2}-\{u,v\}$ does not contain a Hamilton path, then there exist two vertices $u',v'\in V(G_{2})$ such that $G_{2}-\{u',v'\}$ contains a Hamilton path and the following holds:
 $d_{G_{2}}(u')+d_{G_{2}}(v')\leq n_{1} $, $e(\{u',v'\},V(G_{1}))\geq n_{2}+2$.
\end{lem}

The following lemma was obtained independently by Ainouche and Christofides \cite{Ainouche}, Jung \cite{Jung}, Nara \cite{Nara}, and Schmeichel and Hayes \cite{Schmeichel}.
\begin{lem}\label{lem 2.7}\cite{Ainouche,Jung,Nara,Schmeichel}
 Let $G$ be a graph of order $n\geq3$. If $\sigma_{2}(G)\geq n-1$, then one of the following holds:
 \begin{description}
   \item[(i)] $G$ contains a Hamilton cycle.
   \item[(ii)] $K_{m,m+1}$ $\subseteq G \subseteq K_{m}+(m+1)K_{1}$, where $m=\frac{n-1}{2}$ and $n$ is odd with $n\geq5$.
   \item[(iii)] $G\cong K_{1}+(K_{p}\cup K_{q})$ for some positive integers $p$ and $q$ with $p+q=n-1$.
 \end{description}
\end{lem}

\begin{lem}\label{lem 2.6}
Let  $n_{1}$ and $n_{2}$ be two integers with $n_{1}, n_{2}\geq5$ and $G$ a graph on $n=n_{1}+ n_{2}$ vertices. If $\sigma_{2}(G)\geq n+2$ and $G$ has no pair of two disjoint cycles of length $n_1$ and $n_2$, then $V(G)$ has a partition $V(G)=V_1\cup V_2$ such that $|V_1|=n_{1}-2$ and $G[V_1]$ has a Hamilton path; $|V_{2}|=n_{2}+2$,  $\sigma_{2}(G[V_2])\geq n_{2}+3$ and moreover, one of the following holds:
\begin{description}
  \item[(i)] For any two vertices $x,y\in V_2$, $G[V_{2}]-\{x,y\}$ has a Hamilton cycle;
   \item[(ii)] $K_{m+2,m+1}$ $\subseteq G[V_{2}]\subseteq K_{m+2}+(m+1)K_{1}$, where  $n_{2}$ is odd and $m=\frac{1}{2}(n_{2}-1)$;
   \item[(iii)] $3K_{1}+(K_{p}\cup K_{q})\subseteq G[V_{2}]\subseteq K_{3}+(K_{p}\cup K_{q})$, where $p$ and $q$ are two positive integers with $p+q=n_{2}-1$.
 \end{description}
\end{lem}

\begin{proof}
Let $W_1\cup W_2$ be a partition of $V(G)$ with $|W_1|=n_1,|W_2|=n_2$ such that $G[W_i]$ contains a Hamilton path for each $i\in\{1,2\}$ and $e(W_{1})+e(W_{2})$ is maximum among all such partitions. Since $\sigma_{2}(G)\geq n+2$, Theorem~\ref{thm 1.2} implies that $G$ is Hamiltonian and therefore, the partition above exists. Since $G$ has no pair of  two disjoint cycles of length $n_1$ and $n_2$, one of $G[W_1]$ and $G[W_2]$, say $G[W_1]$, is not Hamiltonian. Let $s$ and $t$ be the two endvertices of a Hamilton path in $G[W_1]$.  Since $G[W_1]$ is not Hamiltonian, $s$ and $t$ are not adjacent and, moreover, by Lemma~\ref{lem 2.1} we have
 $d_{G[W_{1}]}(s)+d_{G[W_{1}]}(t)\leq n_{1}-1$ and, therefore,
 \begin{equation}\label{1}
 e(\{s,t\},W_2)\geq n_{2}+3
 \end{equation}
because $\sigma_2(G)\geq n+2$. Let $V_{1}=W_1\setminus\{s,t\}$ and $V_{2}=W_2\cup\{s,t\}$. It is clear that $G[V_{1}]$ contains a Hamilton path since $s$ and $t$ are the two endvertices of a Hamilton path in $G[W_1]$. Further, by Lemma~\ref{lem 2.3} and (\ref{1}), $G[V_{2}]$ has a Hamilton path.

\noindent{\bf Claim 1}. If $\sigma_{2}(G[V_2])\leq n_{2}+2$ then $V_2$ has  two vertices $x$ and $y$ such that $G[V_2]-\{x,y\}$ contains a Hamilton path and the following two inequalities hold:
\begin{equation}\label{3}
d_{G[V_2]}(x)+d_{G[V_2]}(y)\leq n_{2}+2\ \ {\rm and}\ \  e(\{x,y\},V_1)\geq n_{1}.
\end{equation}

\begin{proof} Assume first that $G[V_2]$ is not Hamiltonian. Choose $x$ and $y$ to be the two endvertices of a Hamilton path in $G[V_2]$. Since $G[V_2]$ is not Hamiltonian, we have $xy\notin E(G)$ and, by Lemma \ref{lem 2.1}, $d_{G[V_2]}(x)+d_{G[V_2]}(y)\leq n_{2}+1$. Thus, $e(\{x,y\},V_1)\geq n_{1}+1$, as desired. Now assume that $G[V_2]$ is Hamiltonian. Let $x'$ and $y'$ be arbitrary two nonadjacent vertices in $G[V_2]$. Since $\sigma_{2}(G[V_2])\leq n_{2}+2$, we may choose $x'$ and $y'$ to be such that satisfy (\ref{3}). If $G[V_2]-\{x',y'\}$ contains a Hamilton path then we are done by letting $x=x'$ and $y=y'$. Otherwise, in Lemma~\ref{lem 2.5} we replace $G_{1},G_{2}$ and $u,v$ with $G[V_1],G[V_2]$ and $x',y'$, respectively. Then $u',v'$ clearly satisfy our requirement for $x,y$, which completes the proof of the claim.
\end{proof}

Let $x$ and $y$ be defined as in Claim 1 and let $W'_{1}=V_1\cup \{x,y\},W'_{2}=V_2\setminus\{x,y\}$. Then $G[W'_{2}]$ contains a Hamilton path. By Lemma~\ref{lem 2.3} and (\ref{3}), $G[W'_{1}]$ has a Hamilton path. Moreover, recall that $V_1=W_1\setminus\{s,t\},V_2=W_2\cup\{s,t\}$ and therefore, $|W_1|=|W'_1|,|W_2|=|W'_2|$. Further, by (\ref{1}) and Claim 1, if $\sigma_{2}(G[V_2])\leq n_{2}+2$ then
\begin{eqnarray*}\label{diff}
e(W'_{1})+e(W'_{2})&=&e((W_{1}\setminus\{s,t\})\cup \{x,y\})+e((W_2\cup\{s,t\})\setminus\{x,y\})\\
&\geq &e(W_{1})-(d_{G[W_{1}]}(s)+d_{G[W_{1}]}(t))+e(\{x,y\},V_1)\\
&&+e(W_2)+e(\{s,t\},W_2)-(d_{G[V_2]}(x)+d_{G[V_2]}(y))\\
&\geq &e(W_{1})-(n_{1}-1)+n_{1}+ e(W_{2})+(n_{2}+3)-(n_{2}+2)\\
&=&e(W_{1})+e(W_{2})+2.
  \end{eqnarray*}
This contradicts the maximality of $e(W_{1})+e(W_{2})$. Thus, we have $\sigma_{2}(G[V_2])\geq n_{2}+3$.

Suppose now that (i) does not hold, that is, there are two vertices $x,y\in V_2$ such that $G[V_2]-\{x,y\}$ is not Hamiltonian. Consider the graph $H=G[V_2]-\{x,y\}$. Recalling that $\sigma_{2}(G[V_2])\geq n_{2}+3$, we have $\sigma_{2}(H)\geq n_2-1=|V(H)|-1$. So by Lemma~\ref{lem 2.7} we have $K_{m,m+1}\subseteq H\subseteq K_{m}+(m+1)K_{1}$,
where  $n_{2}$ is odd and $m=\frac{1}{2}(n_{2}-1)$; or $ H\cong K_1+(K_{p}\cup K_{q})$ for some positive integers $p$ and $q$ with $p+q=n_{2}-1$. For the former case, we notice that $H\subseteq K_{m}+(m+1)K_{1}$ means that $H$ contains at least $m+1$ pairwise nonadjacent vertices, which is of course the case for $G[V_2]$. Further, the condition $\sigma_{2}(G[V_2])\geq n_2+3$ implies that each of these $m+1$ vertices is adjacent to both $x$ and $y$. Thus, $K_{m+2,m+1}\subseteq G[V_2]\subseteq K_{m+2}+(m+1)K_{1}$. If the latter holds, it is clear that $3K_{1}+(K_{p}\cup K_{q})\subseteq G[V_2]\subseteq K_{3}+(K_{p}\cup K_{q})$ for some positive integers $p$ and $q$ with $p+q=n_{2}-1$. Thus, (ii) or (iii) holds, which completes our proof.
\end{proof}

\section{Proof of Theorem~\ref{thm 1.5}}

The main idea in our proof comes from \cite{Zahar} and \cite{Yan}. In the following, when we say that a graph $G$ has {\it a pair of disjoint $(n_1,n_2)$-cycles} we always mean that $G$  has two disjoint cycles of length $n_1$ and $n_2$. Further, for more clarity  we always write a cycle $C=v_1v_2\cdots v_n$ as the form $v_1v_2\cdots v_nv_1$, instead of its standard form. 

Firstly, we prove the following proposition.

\noindent{\bf Proposition 1}. If $\sigma_{2}(G)\geq n+2$, then $G$ has a pair of disjoint $(3,n-3)$-cycles and a pair of disjoint $(4,n-4)$-cycles.
\begin{proof}
Since $\sigma_{2}(G)\geq n+2$, by Theorem~\ref{thm A}, $G$ has a 3-cycle, denoted by $C=uvwu$. Let $G'=G-\{u,v,w\}$. Obviously, $\sigma_{2}(G')\geq (n+2)-6=|V(G')|-1$. By Lemma~\ref{lem 2.7}, we distinguish among three cases.

\textbf{Case 1.} $G'$ contains a Hamilton cycle.

 In this case, we need only to find a pair of disjoint $(4,n-4)$-cycles in $G$ for $n\geq 8$.
 Let $C'$ be a Hamilton cycle in $G'$ (with a given direction).

\textbf{Subcase 1.1.} $G'$ has two nonadjacent vertices $x$ and $y$ with $d_{G'}(x)+d_{G'}(y)\leq n-2$.

In this case, $e(\{x,y\}, V(C))\geq 4$. This implies that $x$ or $y$ is adjacent to at least two vertices on $C$, say $uy, vy \in E(G)$. It is clear that $u,v,w,y$ form a 4-cycle. If $y^{+}y^{-}\in E(G)$, then $G'-y$ has an $(n-4)$-cycle. Now assume that $y^{+}y^{-}\notin E(G)$. If $d_{G'}(y^{+})+d_{G'}(y^{-})\geq n-2$ then by Lemma~\ref{lem 2.1}, $G'-y$ is Hamiltonian.  Now we consider the case that $d_{G'}(y^{+})+d_{G'}(y^{-})\leq n-3$. In this case, again by $\sigma_{2}(G)\geq n+2$ we have $e(\{y^-,y^+\},V(C))\geq 5$. This means that at most one of $uy^{-},vy^{-},wy^{-},uy^{+},vy^{+},wy^{+}$ is not an edge in $E(G)$, say $uy^-\notin E(G)$.  Replacing the roles $y^-,y^+$ by $y,y^{2+}$ and repeating the discussion above, we can see that at most one of $uy,vy,wy,uy^{2+},vy^{2+},wy^{2+}$ is not an edge in $E(G)$. Thus,  $zC^{-}[y^{-},y^{2+}]z$ is an $(n-4)$-cycle and $(\{u,v,w\}\setminus\{z\})\cup\{y,y^+\}$ induces a 4-cycle, where $z$ is a common neighbour of $y^-$ and $y^{2+}$ on $C$.

\textbf{Subcase 1.2.} $d_{G'}(x)+d_{G'}(y)\geq n-1$ for any two nonadjacent vertices $x,y\in V(G')$.

We first assume that $G'$ contains a vertex $z$ that is adjacent to at least two vertices on $C$, say $uz,vz \in E(G)$. If $z^-$ and $z^+$ are adjacent then we are done. If $z^-$ and $z^+$ are not adjacent,  then by Lemma~\ref{lem 2.1}, $G'-z$ is Hamiltonian since $d_{G'}(z^-)+d_{G'}(z^+)\geq n-1$, i.e., $d_{G'-z}(z^-)+d_{G'-z}(z^+)\geq (n-1)-2>|V(G'-z)|$. Thus, we get a pair of disjoint $(4,n-4)$-cycles $uwvzu$ and the Hamilton cycle in $G'-z$.

 We now assume that every vertex in $V(G')$ is adjacent to at most one vertex on $C$. Let $z$ be an arbitrary vertex in $G'$ and, without loss of generality, assume that $uz,vz\notin E(G)$. If $d_{G'}(z)\leq \frac{1}{2}(n+1)$, then $d_{G}(u)\geq \frac{1}{2}(n+1)$ and $d_{G}(v)\geq \frac{1}{2}(n+1)$ since $\sigma_{2}(G)\geq n+2$. This means that $u$ and $v$ have a common neighbour in $G'$, which contradicts  that $e(z,C)\leq 1$. Thus,  $\delta(G')\geq\frac{1}{2}(n+2)$.
Let $ux$ be an arbitrary edge with $x\in V(G')$. Since  $e(x,C)\leq 1$, $x$ and $v$ are not adjacent and therefore, have a common neighbour $y$ in $G'-x$.
Notice that $u$ and $v$ are the only neighbours of $x$ and $y$ on $C$, respectively. Therefore, $w$ is neither adjacent to $x$ nor to $y$. So again by $\sigma_{2}(G)\geq n+2$, $w$ is adjacent to at least two vertices $x'$ and $y'$ in $G'-\{x,y\}$. Further, since $\delta(G')\geq\frac{1}{2}(n+2)$, so by Lemma \ref{lem 2.4}, $V(G')\setminus\{x,y\}$ is Hamilton-connected. Thus, we get a pair of disjoint $(4,n-4)$-cycles $vuxyv$ and $wPw$, where $P$ is a Hamilton path in $G'-\{x,y\}$ with endvertices $x'$ and $y'$.

\textbf{Case 2.} $K_{m,m+1}$ $\subseteq G' \subseteq K_{m}+(m+1)K_{1}$, where $m=\frac{1}{2}(n-4)$ and $n\geq8$ is even.

In this case $G'$ has a set of $m+1$ pairwise non-adjacent vertices. We denote this set by $S$ and denote $T=V(G')\setminus S$. It is clear that $|S|=m+1=\frac{1}{2}(n-2)$ and $|T|=m=\frac{1}{2}(n-4)$.  Since $\sigma_{2}(G)\geq n+2$, each vertex in $S$ must be adjacent to all of the vertices on $C$ and the vertices in $T$. Choose a vertex  $s\in S$ and two vertices $t_1,t_2\in T$.  If $t_1$ and $t_2$ are not adjacent then, again by $\sigma_{2}(G)\geq n+2$, either $t_1$ or $t_2$ is adjacent to one of $u,v,w$, say $wt_1\in E(G)$. In either two cases that $t_1$ and $t_2$ are adjacent or not, one can see that $\{u,v,s\}$ forms a 3-cycle and $V(G)\setminus\{u,v,s\}$ has an $(n-3)$-cycle while $\{u,v,w,s\}$  forms a 4-cycle and $V(G)\setminus\{u,v,w,s\}$ has an $(n-4)$-cycle.

\textbf{Case 3.} $G'\cong K_{1}+(K_{p}\cup K_{q})$ for two positive integers $p$ and $q$ with $p+q=n-4$.

 Without loss of generality, we assume that $p\geq q$. If $n=6$, then $G'$ is a path $P=xyz$ where $y\in V(K_1)$. Since $\sigma_{2}(G)\geq n+2$, $y$ is adjacent to two vertices in $\{u,v,w\}$, say $vy,wy\in E(G)$. Thus, $G$ has a pair of disjoint $(3,3)$-cycles $vxyv$ and $uwzu$. We now consider the case that $n\geq7$. Again by the assumption that $\sigma_{2}(G)\geq n+2$, each vertex in $K_p$ and $K_q$ is adjacent to all of the vertices $u,v,w$ because the vertices between $K_p$ and $K_q$ are not adjacent. In this case, it is not difficult to find a pair of disjoint $(3,n-3)$-cycles and a pair of disjoint $(4,n-4)$-cycles in $G$.
The proposition follows.
\end{proof}

By Proposition 1, we need only to consider the case that $n_1,n_2\geq 5$. Throughout the following, we assume that $V(G)$ has a partition $V(G)=V_1\cup V_2$ such that
\begin{description}
 \item[1)]  $|V_1|=n_{1}-2$ and $G[V_1]$ has a Hamilton path;
 \item[2)] $|V_{2}|=n_{2}+2$,  $\sigma_{2}(G[V_2])\geq n_{2}+3$ and therefore, $G[V_2]$ has  a Hamilton cycle, denoted by $C_2$.
\end{description}
Thus, to prove Theorem \ref{thm 1.5}, it suffices to prove that if $G$ satisfies any one of (i),(ii),(iii) in Lemma~\ref{lem 2.6} then $G$ has a pair of disjoint $(n_{1},n_{2})$-cycles. We will distinguish among three claims. The following two propositions are necessary for our further discussion.

\noindent{\bf Proposition 2}. $G[V_1]$ is Hamiltonian or $G[V_1\cup\{x,y\}]$ is  Hamiltonian for some two consecutive vertices $x$ and $y$ on $C_2$.
\begin{proof} Let $u$ and $v$ be the two endvertices of a Hamilton path in $G[V_1]$. If $u$ and $v$ are adjacent then clearly $G[V_1]$ is Hamiltonian. If $d_{G[V_{1}]}(u)+d_{G[V_{1}]}(v)\geq n_{1}-2$ then by Lemma~\ref{lem 2.1}, again $G[V_{1}]$ is Hamiltonian. Now assume that $u$ and $v$ are not adjacent and $d_{G[V_{1}]}(u)+d_{G[V_{1}]}(v)\leq n_{1}-3$. In this case we have $e(\{u,v\},V_{2})\geq n_{2}+5$, which implies that there are two consecutive vertices $x$ and $y$ on $C_{2}$ such that $e(\{u,v\},\{x,y\})$ $\geq 3$. Thus, $G[V_{1}\cup \{x,y\}]$ is Hamiltonian.
\end{proof}
\noindent{\bf Proposition 3}. Let $u\in V_1$ and $x,y\in V_2$. If $G[V_{1}\cup \{x,y\}]$ is not Hamiltonian but has  a Hamilton path with $u$ and $x$ as the two endvertices, then $u$ and $x$ have three common neighbours in $V_{2}\setminus\{x,y\}$.
\begin{proof}
Let $\Gamma=G[V_{1}\cup \{x,y\}]$. If $ux\in E(G)$, or $ ux\notin E(G)$ but $d_{\Gamma}(u)+d_{\Gamma}(x)\geq n_{1}$ then $\Gamma$ would be Hamiltonian because of Lemma~\ref{lem 2.1}. Now assume that $ ux\notin E(G)$ and $d_{\Gamma}(u)+d_{\Gamma}(x)\leq n_{1}-1$. Then $e(\{u,x\},V_{2}\setminus\{x,y\})\geq n_{2}+3$. This means that $u$ and $x$ have at least three common neighbours in $V_{2}\setminus\{x,y\}$.
\end{proof}

\noindent{\bf Claim 1}. If  $G[V_{2}]-\{x,y\}$ is Hamiltonian for any two vertices $x,y\in V_{2} $ then $G$ has a pair of disjoint $(n_{1},n_{2})$-cycles.
\begin{proof}
It suffices to prove that $G[V_1\cup\{x,y\}]$ is Hamiltonian for some $x,y\in V_2$. If $G[V_1\cup\{x,y\}]$ is  Hamiltonian for two consecutive vertices $x$ and $y$ on $C_2$ then the claim follows directly. Now assume that $G[V_1\cup\{x,y\}]$ is not Hamiltonian for any two consecutive vertices $x$ and $y$ on $C_2$. By Proposition 2, let $C_{1}$ be a Hamilton cycle in $G[V_{1}]$.

\textbf{Case 1.} $\sigma_{2}(G[V_{1}])< n_{1}$.

Let $u$ and $v$ be two nonadjacent vertices in $V_1$ and $d_{G[V_{1}]}(u)+d_{G[V_{1}]}(v)< n_{1}$. Since $\sigma_2(G)\geq n+2$, we have $e(\{u,v\},V_{2})\geq n_{2}+3$. This implies that $C_{2}$ has two consecutive vertices $x$ and $y$ such that $e(\{u,v\},\{x,y\})\geq 3$ and therefore, $x$ or $y$ is adjacent to both $u$ and $v$, say $xu,xv\in E(G)$. Since $G[V_1\cup\{x,y\}]$ is not Hamiltonian, $\{u,v\}$ is not the pair of the two endvertices of any Hamilton path in $G[V_1]$. This implies that $u^{-}v^{-}\notin E(G)$ and, moreover, by Lemma~\ref{lem 2.2}, $d_{G[V_{1}]}(u^{-})+d_{G[V_{1}]}(v^{-})\leq n_{1}-2$. Thus, $e(\{u^{-},v^{-}\},V_{2})\geq n_{2}+4$ and therefore, $u^-$ and $v^-$ have a common neighbour $x'$ in $V_2$ with $x'\not=x$. So $C_{1}^{-}[u^{-},v]xC_{1}[u,v^{-}]x'u^{-}$ is a Hamilton cycle in $G[V_1\cup\{x,x'\}]$, again as desired.

\textbf{Case 2.} $\sigma_{2}(G[V_{1}])\geq n_{1}$.

Notice that $n_{1}=|V_{1}|+2$. So by Lemma~\ref{lem 2.4}, $G[V_{1}]$ is Hamilton-connected.

Let $ux\in E(G)$ with $u\in V(C_{1})$ and $x\in V(C_2)$. Then $G[V_{1}\cup \{x,x^{+}\}]$ has a Hamilton path with two  endvertices $u^{+}\in V(C_1)$ and $x^{+}\in V(C_2)$. Since $G[V_{1}\cup \{x,x^{+}\}]$ is not Hamiltonian, then by Proposition 3, $u^+$ and $x^+$  have a common neighbour  $y\in V_{2}\setminus\{x,x^{+}\}$, i.e., $x^{+}y,u^{+}y\in E(G)$. Since $\sigma_{2}(G[V_{1}])\geq n_{1}=|V_{1}|+2$,  we have $d_{G[V_{1}]-u}(w)+d_{G[V_{1}]-u}(w')\geq n_{1}-2=|V_{1}\setminus\{u\}|+1$ for any two nonadjacent vertices $w$ and $w'$ in $V_1\setminus\{u\}$. Thus, by Lemma~\ref{lem 2.4}, $G[V_{1}]-u$ is Hamilton-connected.

If $e(x,V_{1})\geq2$, then $x$ is adjacent to a vertex $v$ in $V_1$ other than $u$. Moreover, we may choose a direction of $C_{1}$ such that $v\neq u^{+}$. Then $G[V_{1}]-u$ has a Hamilton path $P$ with endvertices $v$ and  $u^{+}$ because $G[V_{1}]-u$ is Hamilton-connected. Thus, $G[V_{1}\cup \{x,y\}]$ has a Hamilton path $uxPy$ with endvertices $u$ and $y$. If $G[V_{1}\cup \{x,y\}]$ is Hamiltonian, then we are done. If not, by Proposition 3, $u$ and $y$ have a common neighbour $z\in V_{2}\setminus\{x,y\}$. Thus, $G[V_{1}\cup \{z,y\}]$ has a Hamilton cycle $yC_{1}[u^{+},u]zy$.

Now consider the case that $e(x,V_{1})\leq 1$ for any vertex $x\in V_{2}$. First assume that there is a vertex $x\in V_{2}$ such that $d_{G[V_{2}]}(x)\leq \lfloor \frac{n_{2}+5}{2}\rfloor$, i.e. $d_G(x)\leq \lfloor \frac{n_{2}+5}{2}\rfloor +1$. Let $u$ and $v$ be two distinct vertices in $G[V_{1}]$ such that $ux\notin E(G)$, $vx\notin E(G)$. By the degree sum condition, we have $d(u)+d(v)+2d_G(x)\geq 2(n+2)$£¬ and therefore, $d(u)+d(v)\geq 2n_{1}+n_{2}-3$. Since $|V_{1}|=n_{1}-2$, we have $d_{G[V_{1}]}(u)+d_{G[V_{1}]}(v)\leq 2(n_{1}-3)$, and then $e(\{u,v\},V_{2})\geq n_{2}+3$. Thus, there is a vertex $y\in V_{2}$ such that $uy,vy\in E(G)$, a contradiction.

Assume now that $d_{G[V_{2}]}(x)\geq \lfloor \frac{n_{2}+5}{2}\rfloor +1\geq \frac{n_{2}+6}{2}$ for any vertex $x\in V_{2}$. Let $ux\in E(G)$ with $u\in V(C_{1})$ and $x\in V(C_{2})$. Then $u^{+}$ and $x^{+}$ are the endvertices of the Hamilton path $C_{1}[u^{+},u]xx^{+}$ in $G[V_{1}\cup \{x,x^{+}\}]$. If $G[V_1\cup\{x,x^+\}$ is not Hamiltonian then by Proposition 3, $u^+$ and $x^+$ have three common neighbours, say one of which is $w$. Similarly, $u^-$ and $x^+$ have three common neighbours. Further, notice that $C^-_1[u,u^+]wx^+$ is a Hamilton path with endvertices $u$ and $x^+$. If $G[V_1\cup\{w,x^+\}]$ is not Hamiltonian then $u$ and $x^+$ have three common neighbours. Consequently, we can choose three distinct vertices $y_1,y_2,y_3\in V_2$ such that $u^+y_1,uy_2,u^-y_3,x^+y_1,x^+y_2,x^+y_3\in E(G)$. Thus, $G[(V_{1}-u)\cup \{y_1,x^{+},y_3\}]$ has a Hamilton cycle $C_{1}[u^{+},u^{-}]y_3x^{+}y_1u^{+}$.
Write $\Gamma=G[V_{2}- \{y_1,x^{+},y_3\}]$. Notice that, for any pair of nonadjacent vertices $a$ and $b$ in $\Gamma$, we have $d_{\Gamma}(a)+d_{\Gamma}(b)\geq (n_{2}+6)-6=|\Gamma|+1$, i.e., $\sigma_{2}(\Gamma)\geq |\Gamma|+1$. So by Lemma~\ref{lem 2.4}, $\Gamma$ is Hamilton-connected. Thus, $\Gamma$ has a Hamilton path $P$ with endvertices $x$ and $y_2$, and hence we obtain a desired $n_2$-cycle $uPu$.
\end{proof}

\noindent{\bf Claim 2}.  If $K_{m+2,m+1}$ $\subseteq G[V_{2}]\subseteq K_{m+2}+(m+1)K_{1}$, then $G$ has a pair of disjoint $(n_{1},n_{2})$-cycles.
\begin{proof}
It is clear that $G[V_{2}]$ has a set of $m+1$ pairwise non-adjacent vertices. We denote this set by $S$ and denote $T=V_{2}\setminus S$. Thus, $|S|=m+1=\frac{1}{2}(n_{2}+1)\geq 3$ and $|T|=m+2=\frac{1}{2}(n_{2}+3)\geq 4$. Since $\sigma_{2}(G[V_{2}])=n_{2}+3$, $G[T]$ has at most one isolated vertex.

Let $P$ be a Hamilton path in $G[V_{1}]$ and let $u$ and $v$ be its two endvertices.
If $d_{G[V_{1}]}(u)+d_{G[V_{1}]}(v)\leq n_{1}-3$ and $uv\notin E(G)$, then $e(\{u,v\},V_{2})\geq n_{2}+5$, say $e(u,V_{2})\geq \frac{1}{2}(n_{2}+5)$. Thus, $e(u,T)\geq2$ and $e(u,S)\geq1$, say $us,ut_{1},ut_{2}\in E(G)$ where $s\in S, t_{1},t_{2}\in T$. Since $G[T]$ has at most one isolated vertex, one of $G[T]-t_{1}$ and $G[T]-t_{2}$ contains an edge, say $E(G[T]-t_{1})\not=\emptyset$. If there is a vertex $s_{1}\in S$ such that $vs_{1}\in E(G)$, then $G$ has a pair of disjoint $n_{1}$-cycles $t_{1}Ps_{1}$ and an $n_{2}$-cycles in $G[V_{2}]-\{t_{1},s_{1}\}$. Otherwise, we have $e(v,S)=0$ and therefore, $e(v,T)\geq 3$ because  $e(\{u,v\},V_{2})\geq n_{2}+5$. Thus, there is a vertex $t\in T$ such that $E(G[T]-t)$ is not empty. Hence, $G$ has a pair of disjoint $n_{1}$-cycle $sPt$ and an $n_{2}$-cycle in $G[V_{2}]-\{t,s\}$.

We now assume that $uv\in E(G)$ or $d_{G[V_{1}]}(u)+d_{G[V_{1}]}(v)\geq n_{1}-2$ holds for the two endvertices $u,v$ of any Hamilton path in $G[V_{1}]$. Similar to the proof of Proposition 2, $G[V_{1}]$ is Hamiltonian. Let $C_{1}=v_1v_2\cdots v_{n_{1}-2}v_{1}$ be a Hamilton cycle in $G[V_{1}]$.

Since $\sigma_{2}(G[V_{2}])=n_{2}+3$ and $|S| \geq3$, there are two distinct vertices $s_{1},s_{2}\in S$ such that $e(s_{1},V_{1})\geq\frac{1}{2} (n_{1}-1)$ and $e(s_{2},V_{1})\geq\frac{1}{2} (n_{1}-1)$. Thus, each of $s_{1}$ and $s_{2}$ is adjacent to a pair of two  consecutive vertices on $C_{1}$, say $\{v_{i},v_{i+1}\}$ and $\{v_{j},v_{j+1}\}$, respectively.

\noindent{\bf Claim 2.1}. If $G$ has no pair of disjoint $(n_{1},n_2)$-cycles and $G[T]$ contains a $P_{4}$ then the pairs $\{v_{i},v_{i+1}\}$ and $\{v_{j},v_{j+1}\}$ can be properly chosen to be distinct.
\begin{proof} If $s_1$ or $s_2$ is adjacent to at least two pairs of two  consecutive vertices on $C_{1}$ then the assertion clearly holds. Assume now that each of $s_1$ and $s_2$ is adjacent to exactly one pair of two  consecutive vertices and moreover, they are adjacent to the same pair, say $\{v_1,v_2\}$. In this case, recall that $e(s_{1},V_{1})\geq\frac{1}{2} (n_{1}-1)$ and $e(s_{2},V_{1})\geq\frac{1}{2} (n_{1}-1)$. This implies that $s_iv_{4},s_iv_{6},\cdots, s_iv_{n_{1}-3} \in E(G)$ and $s_iv_{3},s_iv_{5},\cdots, s_iv_{n_{1}-2} \notin E(G)$  for each $i\in\{1,2\}$ (in this case $n_1$ must be odd). Since $G[T]$ contains a $P_{4}$, $G[T]-\{t\}$ has at least one edge for any $t\in T$.  Thus, if $tv_{3}\in E(G)$ for some $t\in T$ then $G$ has a pair of disjoint $n_{1}$-cycle $s_1C_{1}[v_{4},v_{3}]t$ and an $n_{2}$-cycle in $G[V_{2}]-\{t,s_1\}$, a contradiction. Similarly, if $v_{3}$ is adjacent to some $v_i$ with $i$  odd then $G$ has a pair of disjoint $n_{1}$-cycle $s_{1}v_2s_{2}C^{-}_{1}[v_{i-1},v_{3}]v_iC^{-}_{1}[v_{i},v_{1}]s_{1}$ and an $n_{2}$-cycle in $G[V_{2}]-\{s_{1},s_{2}\}$. This is again a contradiction. Therefore, $v_3$ is adjacent to neither vertex in $T$ nor vertex in $\{v_{5},v_{7},\cdots, v_{n_{1}-2}\}$. Therefore,
$$d_{G}(s_1)+d_{G}(v_{3})\leq \left(\frac{1}{2} (n_{1}-1)+\frac{1}{2} (n_{2}+3)\right)+\left(\frac{1}{2} (n_{1}-3)+\frac{1}{2} (n_{2}+1)\right)< n+2,$$
which contradicts our assumption that $\sigma_2(G)\geq n+2$ since $s_1$ and $v_3$ are not adjacent.
   \end{proof}

If $G[T]$ contains a $P_{4}$ then by Claim 2.1,  $G[V_1\cup\{s_1,s_2\}]$ is Hamiltonian and moreover, it is not difficult to find a  Hamilton cycle in $G[V_2]-\{s_1,s_2\}$.

If $G[T]$ does not contain $P_{4}$, then each component in $G[T]$ is a star (i.e., $K_{1,q}$) or has at most three vertices.  Thus, we may choose two nonadjacent vertices $t_{1},t_{2}\in T$ such that $d_{G[T]}(t_{1})\leq 2$ and $d_{G[T]}(t_2)\leq 2$. Moreover, each of $G[T]-t_1$ and $G[T]-t_2$ has at least one edge  because, except one possible isolated vertex, every vertex in $G[T]$ has degree at least one.  Therefore, $d_{G[V_{2}]}(t_{1})+d_{G[V_{2}]}(t_{2})\leq n_{2}+5$ and hence, $e(\{t_{1},t_2\},V_{1})\geq n_{1}-3$, say $e(t_{1},V_{1})\geq \frac{1}{2} (n_{1}-3)$.  Choose a Hamilton cycle $C_{2}$ in $G[V_{2}]$ such that $t_{1}^{+},t_{1}^{-}\in S$. Since $d_{G[V_{2}]}(t_{1}^{+})+d_{G[V_{2}]}(t_{1}^{-})=n_{2}+3$, we have $e(\{t_{1}^{+},t_{1}^{-}\},V_{1})\geq n_{1}-1$ and therefore, $t_1^-$ or $t_1^+$ is adjacent to two consecutive vertices on $C_{1}$, say $t_{1}^{-}v_{i},t_{1}^{-}v_{i+1}\in E(G)$. If $t_{1}v_{i-1}\in E(G)$, then $G$ has a pair of disjoint $n_{1}$-cycle $t_{1}C_{1}^{-}[v_{i-1},v_{i}]t_{1}^{-}t_{1}$ and an $n_{2}$-cycle in $G[V_{2}]-\{t_{1},t_{1}^{-}\}$ since $G[T]-t_1$ has at least one edge. Similarly, if $t_1$ is adjacent to one of $v_i,v_{i+1},v_{i+2}$ then we can get a pair of disjoint $(n_{1},n_2)$-cycles. Now assume that  $t_{1}c_{i-1},t_{1}c_{i},t_{1}c_{i+1},t_{1}c_{i+2}\notin E(G)$. Then $t_1$ is adjacent to two consecutive vertices $v_{j},v_{j+1}$ on $C_{1}$ as  $e(t_{1},V_{1})\geq \frac{1}{2} (n_{1}-3)$. Therefore, $G$ has a pair of disjoint $n_{1}$-cycle $t_{1}C_{1}[v_{j+1},v_{i}]t_{1}^{-}C_{1}[v_{i+1},v_{j}]t_{1}$ and an $n_{2}$-cycle in $G[V_{2}]-\{t_{1},t_{1}^{-}\}$.
The claim follows.
\end{proof}

\noindent{\bf Claim 3}. If $3K_{1}+(K_{p}\cup K_{q}) \subseteq G[V_{2}] \subseteq K_{3}+(K_{p}\cup K_{q})$ for some positive integers $p$ and $q$ with $p+q=n_{2}-1$, then $G$ has a pair of disjoint $(n_{1},n_{2})$-cycles.

\begin{proof} If $n_{2}=5$, then $K_{4,3}$ $\subseteq 3K_{1}+(K_{p}\cup K_{q}) \subseteq 3K_{1}+K_{4}$. So by Claim 2, Claim 3 holds. Thus, we assume that $n_{2}\geq6$. Let $C_2$ be a Hamilton cycle in $3K_{1}+(K_{p}\cup K_{q})$ and let $M=3K_{1}$. It is clear that $C_2$ is also a Hamilton cycle in $G[V_2]$. So by Proposition 2,  $G[V_1\cup\{x,y\}]$ is  Hamiltonian for some two consecutive vertices $x$ and $y$ on $C_2$ or $G[V_1]$ is Hamiltonian. Noticing that at most one of $x$ and $y$ is in $M$, $(3K_{1}+(K_{p}\cup K_{q}))\setminus\{x,y\}$ is Hamiltonian. So if the former holds, then we are done. We now assume that $G[V_1]$ has a Hamilton cycle $C_1$.

For any vertex $x\in K_p$ and $y\in K_q$, since $xy\notin E(G)$ and $\sigma_{2}(G[V_{2}])=n_{2}+3$, we have $e(\{x,y\},V_{1})\geq n_{1}-1$. Therefore, $x$ or $y$ is adjacent to two consecutive vertices $w,w^{+}$ on $C_{1}$, say $xw,xw^{+}\in E(G)$. Thus, $G[V_{1}\cup \{x,x^{+}\}]$ has a Hamilton path $C_{1}[w^{2+},w^{+}]xx^{+}$ with endvertices $w^{2+}$ and $x^{+}$. If $G[V_{1}\cup \{x,x^{+}\}]$ is Hamiltonian, then we are done. If not, then by Proposition 3, $w^{2+}$ and $x^+$ have  a common neighbour $z\in V_{2}\setminus\{x,x^{+}\}$. It is clear that $G[V_2]-\{x,z\}$ is Hamiltonian because $x\notin M$. Hence, if $G[V_1\cup\{x,z\}]$ is Hamiltonian then we are done. If not, again by Proposition 3, $w^+$ and $z$ have a common neighbour in $V_2\setminus\{x,z\}$, say $t\in V_2\setminus\{x,z\}$. Noticing that at most one of $z$ and $t$ is in $M$, $G[V_2]-\{z,t\}$ is Hamiltonian. Thus, we get a pair of an $n_1$-cycle $C_1[w^{2+},w^+]tzw^{2+}$ and an $n_2$-cycle (Hamilton cycle) in $G[V_2]-\{z,t\}$. \end{proof}

By Lemma \ref{lem 2.6} and the three claims above, Theorem \ref{thm 1.5} follows.

\noindent\textbf{Remark.} Since $\sigma_2(G)\leq \delta(G)$, Theorem~\ref{thm 1.5} gives a generalization of Theorem~\ref{thm 1.3} when $n_{1}$ and $n_{2}$  are both odd. This remains a natural question:  Can $\sigma_{2}(G)\geq n+1$ guarantee that $G$ has a pair of disjoint $(n_1,n_2)$-cycles if at least one of $n_1$ and $n_2$ is even?

 \section{Acknowledgements}
This work was supported by the National Natural Science Foundation of China [Grant numbers, 11471273, 11561058].


\begin{thebibliography}{15}

\bibitem{Abbasi}  S. Abbasi, The solution of the El-Zahar problem, Doctoral dissertation, Rutgers University, 1998.

\bibitem{Ainouche}  A. Ainouche, N. Christofides, Condition for the existence of Hamiltonian circuits in graphs based
on vertex degrees, J. Lond. Math. Soc. 32 (1985) 385-391.

\bibitem{Bondy}  J.A. Bondy and U.S.A. Murty, Graph theory with applications, Elsevier, New York (1976).

\bibitem{Bondy2}  J.A. Bondy, Pancyclic Graph I, J. Combin. Theory, 11 (1971) 80-84.

\bibitem{Dirac} G.A. Dirac, Some theorems on abstract graphs, Proc. Math. Soc., 2 (3) (1952) 69-81.

\bibitem{Zahar} M.H. El-Zahar, On circuits in graphs, Discrete Math., 50 (1984) 227-230.

\bibitem{E} P. Erd\H{o}s, Some Recent Combinatorial Problems, Technical report, University of bielefeld, 1990.

\bibitem{P} P. Erd\H{o}s, J.Gallai, On maximal paths and circuits of graphs, Acta Math. Acad. Sci.Hung., 10 (1959) 337-356.

\bibitem{Jung} H.A. Jung, On maximal circuits in finite graphs, Ann. Discrete Math., 3 (1978) 129-144 .

\bibitem{Kostochka} A.V. Kostochka, G. Yu, Graphs containing every 2-factor, Graph Combin., 28 (2012) 687-716.

\bibitem{Nara}  C. Nara, On sufficient conditions for a graph to be hamiltonian, Nat. Sci. Rep. Ochanomizu Univ.
31 (1980) 75-80 .

\bibitem{Ore} O. Ore, Note on Hamiltonian circuits, Amer. Math. Mon. 67 (1960) 55.

\bibitem{Ore2} O. Ore, Theory of graphs, American Mathematical Society, Providence, RI, 1962.

\bibitem{Schmeichel}  E. Schmeichel, D. Hayes, Some extensions of Ore's theorem. In: Alavi, Y. (ed.), Graph theory and
its applications to algorithms and computer science, (1985) 687-695.

\bibitem{Wang} H. Wang, Proof of the Erd\H{o}s-Faudree conjecture on quadrilaterals, Graph Combin., 26 (2010) 833-877.

\bibitem{Wang2} H. Wang, Covering a subset with two cycles, Austra. J. Combin., 65 (1) (2016) 27-36.

\bibitem{Yan} J. Yan, S. Zhang, Y. Ren, J. Cai, Degree sum conditions on two disjoint cycles in graphs, Inf. Process. Lett., 138 (2018) 7-11.

\end{thebibliography}
\end{document}